\documentclass[12pt,reqno]{amsart}
\usepackage{amssymb}
\usepackage{enumerate}
\usepackage[T1]{fontenc}
\usepackage[all]{xy}
\usepackage[usenames]{color}

\newcommand{\bmid}{\,\big|\,}    
\newcommand{\linkset}{\mathfrak{Link}}

\let\Gamma=\varGamma

\newlength{\short}
\let\too=\longrightarrow
\let\xto=\xrightarrow

\newcommand{\higherlim}[3]{H^{#1}(#2;#3)}

\let\oldcirc=\circ
\renewcommand{\circ}{\mathchoice
    {\mathbin{\scriptstyle\oldcirc}}{\mathbin{\scriptstyle\oldcirc}}
    {\mathbin{\scriptscriptstyle\oldcirc}}
    {\mathbin{\scriptscriptstyle\oldcirc}}}

\SelectTips{cm}{10} \UseTips   

\mathcode`\:="003A
\mathchardef\cdot="0201

\renewenvironment{enumerate}[1][]
{\begin{enumerat}[#1]\setlength{\itemsep}{6pt}}{\end{enumerat}}

\newenvironment{enuma}{\begin{enumerate}[{\rm(a) }]}{\end{enumerate}}
\newenvironment{enumi}{\begin{enumerate}[{\rm(i) }]}{\end{enumerate}}

\renewenvironment{itemize}
{\begin{itemiz}\setlength{\itemsep}{6pt}\setlength{\itemindent}{-20pt}}
{\end{itemiz}}

\def\beq#1\eeq{\begin{equation*}#1\end{equation*}}
\def\beqq#1\eeqq{\begin{equation}#1\end{equation}}

\let\emptyset=\varnothing

\DeclareMathAlphabet\EuR{U}{eur}{m}{n}
\SetMathAlphabet\EuR{bold}{U}{eur}{b}{n}
\newcommand{\curs}{\EuR}
\newcommand{\Ab}{\curs{Ab}}
\renewcommand{\mod}{\textup{-}\curs{mod}}

\renewcommand{\:}{\colon}

\newlength{\upto}\newlength{\dnto}
\newcounter{let} 
\loop\stepcounter{let}
\expandafter\edef\csname cal\alph{let}\endcsname%
{\noexpand\mathcal{\Alph{let}}}
\ifnum\thelet<26\repeat

\newcommand{\tdef}[2][]{\expandafter\newcommand\csname#2\endcsname%
{#1\textup{#2}}}
\tdef{Iso}   \tdef{Aut}    \tdef{Out}    \tdef{Inn}    \tdef{Hom}
\tdef{End}   \tdef{Inj}    \tdef{map}    \tdef{Ker}    \tdef{Ob}
\tdef{Mor}   \tdef{Res}    \tdef{Spin}   \tdef{Id}     \tdef{Fr}
\tdef{rk}    \tdef{conj}   \tdef{incl}   \tdef{proj}   \tdef{diag} 
\tdef{trf}   \tdef{Sol}    \tdef{He}     \tdef{Sz}     \tdef{cj}
\tdef{free}  \tdef{ev}     \tdef{Map}    \tdef{Rep}    \tdef{pr}
\tdef{supp}  \tdef{res}    \tdef{hofibre}  \tdef{Vect}  \tdef{hofib}
\tdef{Coker} \tdef{St}
\tdef[_]{fc}  \tdef[_]{fn}  \tdef[_]{typ} \tdef[^]{op}   \tdef[^]{ab}

\newcommand{\SL}{\textit{SL}}

\newcommand{\fdef}[1]{\expandafter\newcommand\csname#1\endcsname%
{\mathfrak{#1}}}
\fdef{X} \fdef{Y} \fdef{N} \fdef{red}  \fdef{foc}  \fdef{hyp}  \fdef{Sub}

\newcommand{\bbdef}[1]{\expandafter\newcommand%
\csname#1\endcsname{\mathbb{#1}}}
\bbdef{C} \bbdef{F} \bbdef{R} \bbdef{Z} \bbdef{Q}

\newcommand{\gen}[1]{\langle{#1}\rangle}

\let\nsg=\normal

\newcommand{\syl}[2]{\textup{Syl}_{#1}(#2)}
\newcommand{\sylp}[1]{\syl{p}{#1}}

\newcommand{\outf}{\Out_{\calf}}

\newcommand{\homf}{\Hom_{\calf}}

\newcommand{\sminus}{\smallsetminus}
\newcommand{\defeq}{\overset{\textup{def}}{=}}

\newtheorem{Thm}{Theorem}[section]
\newtheorem{Prop}[Thm]{Proposition}

\newtheorem{Lem}[Thm]{Lemma}

\newtheorem{Defi}[Thm]{Definition}

\newcommand{\longleft}[1]{\;{\leftarrow%
\count255=0 \loop \mathrel{\mkern-6mu}%
    \relbar\advance\count255 by1\ifnum\count255<#1\repeat}\;}
\newcommand{\longright}[1]{\;{\count255=0 \loop \relbar\mathrel{\mkern-6mu}%
    \advance\count255 by1\ifnum\count255<#1\repeat\rightarrow}\;}
\newcommand{\Right}[2]{\overset{#2}{\longright{#1}}}
\newcommand{\RIGHT}[3]{\mathrel{\mathop{\kern0pt\longright{#1}}
    \limits^{#2}_{#3}}}

\newcommand{\LEFT}[3]{\mathrel{\mathop{\kern0pt\longleft{#1}}\limits^{#2}_{#3}}
}
\newcommand{\longleftright}[1]{\;{\leftarrow\mathrel{\mkern-6mu}%
    \count255=0\loop\relbar\mathrel{\mkern-6mu}%
    \advance\count255 by1\ifnum\count255<#1\repeat\rightarrow}\;}

\newcommand{\onto}[1]{\;{\count255=0 \loop \relbar\joinrel
    \advance\count255 by1
    \ifnum\count255<#1 \repeat \twoheadrightarrow}\;}

\newcommand{\RLEFT}[3]{\mathrel{%
   \mathop{\vcenter{\baselineskip=0pt\hbox{$\kern0pt\longright{#1}$}%
   \hbox{$\kern0pt\longleft{#1}$}}}\limits^{#2}_{#3}}}

\theoremstyle{definition}

\newtheorem{Ex}[Thm]{Example}

\usepackage[usenames]{color}

\usepackage[usenames]{color}                                      %
\newcommand{\xxto}[1]{\mathrel{\mathop{%
  \setbox0\hbox{$\ {\scriptstyle#1}\ $}%
  \hbox to \wd0{\rightarrowfill}}^{#1}}%
}

\newcommand{\xlto}[2][]{%
  \mathrel{\mathop{%
    \setbox0\vbox{
      \hbox{$\scriptstyle\;\;{#1}\;\;$}%
      \hbox{$\scriptstyle\;\;{#2}\;\;$}%
    }%
    \hbox to\wd0{\leftarrowfill}\displaystyle}%
  \limits^{#2}\ifx{#1}{}\else{_{#1}}\fi}%
}

\author{B. Oliver}
\address{Université Paris 13, Sorbonne Paris Cité, LAGA, UMR 7539 du CNRS, 
99, Av. J.-B. Clément, 93430 Villetaneuse, France.}
\email{bob@math.univ-paris13.fr}
\thanks{B. Oliver is partially supported by UMR 7539 of the CNRS}



\title{A remark on the construction of centric linking systems}

\begin{document}

\begin{abstract} 
We give examples to show that it is \emph{not} in general possible to prove 
the existence and uniqueness of centric linking systems associated to a 
given fusion system inductively by adding one conjugacy class at a time to 
the categories. This helps to explain why it was so difficult to prove that 
these categories always exist, and also helps to motivate the procedure 
used by Chermak \cite{Chermak} when he did prove it.
\end{abstract}

\maketitle

When $S$ is a finite $p$-group, a \emph{saturated fusion system} $\calf$ 
over $S$ is a category $\calf$ whose objects are the subgroups of $S$, 
whose morphisms are injective homomorphisms between the subgroups, and 
which satisfies certain axioms originally due to Puig \cite[\S\,2.9]{Puig1} (who 
calls it a Frobenius $S$-category). Equivalent sets of axioms can be 
found, for example, in \cite[Definition 1.2]{BLO2} and \cite[Definition 
I.2.2]{AKO}. We omit the details of those axioms here, except to note that 
if $\varphi$ is a morphism in $\calf$, then all restrictions of $\varphi$ 
(obtained by restricting the domain and/or the target) are also in $\calf$, 
and $\varphi^{-1}$ is in $\calf$ if $\varphi$ is an isomorphism of groups. 
The motivating example is the fusion system $\calf_S(G)$, when $G$ is a 
finite group and $S\in\sylp{G}$. In this case, for $P,Q\le S$, 
	\[ \Mor_{\calf_S(G)}(P,Q) = 
	\bigl\{\varphi\in\Hom(P,Q) \,\big|\, \varphi=c_g = 
	(x\mapsto gxg^{-1}), \textup{ some $g\in G$} \bigr\}. \]

A \emph{centric linking system} associated to a saturated fusion system 
$\calf$ over $S$ is a category $\call$ whose objects are the 
$\calf$-centric subgroups of $S$ (Definition \ref{d:F-centric}), together 
with a functor $\pi\:\call\too\calf$, which satisfy certain conditions 
listed in Definition \ref{d:link}. One of the central questions in the 
field has been that of whether each saturated fusion system does admit an 
associated centric linking system, and if so, whether it is unique up to 
isomorphism. One obvious way to try to construct a centric linking system 
associated to $\calf$ is to do it one $\calf$-conjugacy class (i.e., 
isomorphism class in $\calf$) at a time. One begins with a ``linking 
system'' having as unique object $S$ itself (this is not difficult). One 
then extends the category to also include an isomorphism class of maximal 
$\calf$-centric subgroups of $S$, and continues adding objects until all 
$\calf$-centric subgroups have been included. In this way, the difficulties 
in the construction are split up, and one need only work with one small 
``piece'' of the categories at a time. 

The main result of this paper is to present some simple examples that show 
that this procedure is not possible in general. We construct examples of 
fusion systems $\calf$, and sets $\Y\subseteq\X$ of $\calf$-centric 
subgroups of $S$ which are closed under $\calf$-conjugacy and overgroups 
(and differ by exactly one $\calf$-conjugacy class), such that there is 
more than one isomorphism class of linking systems associated to $\calf$ 
with object set $\Y$, only one of which can be extended to a linking system 
with object set $\X$. A general framework for doing this is given in 
Theorem \ref{Lambda3}, and explicit examples satisfying the hypotheses of 
the theorem are found in Examples \ref{ex1}, \ref{ex2}, and \ref{ex3}.

These examples help to explain why a general construction of centric 
linking systems associated to arbitrary saturated fusion systems was so 
difficult: one cannot expect to find an inductive construction based on 
adding one isomorphism class at a time to the set of objects. They also 
show that the claims by Puig in \cite{Puig} (in the introduction and the 
beginning of \S\,6) that under the above assumptions on $\Y\subseteq\X$, 
there is up to isomorphism a unique $\Y$-linking system associated to 
$\calf$ and it always extends to an $\X$-linking system, are not true. 
(What we call here an ``$\X$-linking system associated to $\calf$'' is 
called a ``perfect $\calf^\X$-locality'' in \cite{Puig}.)

The existence and uniqueness of centric linking systems was shown by 
Chermak \cite{Chermak,O-Ch} in 2011. His proof used the classification of 
finite simple groups, but more recent work by Glauberman and Lynd 
\cite{GLynd} has shown that this dependence can be removed. Chermak's 
construction was also based on an inductive procedure, but he avoided the 
difficulty raised by the examples constructed here by adding (in general) 
several $\calf$-conjugacy classes at a time, and doing so following a very 
precise algorithm. This is just one of several remarkable features of his 
construction.

\noindent\textbf{Notation:} We write $C_X(G)$ for the centralizer of an 
action of $G$ on a set or group $X$; i.e., the elements fixed by $G$. Also, 
$c_a$ denotes \emph{left} conjugation by $a$: $c_a(x)=\9ax=axa^{-1}$. When $\calc$ 
is a small category and $F\:\calc\op\too\Ab$ is a functor to abelian 
groups, $\higherlim{i}{\calc}F$ denotes the $i$-th higher derived functor 
of the inverse limit of $F$.
Whenever $F\:\calc\too\cald$ is a functor and $c,c'\in\Ob(\calc)$, we let 
$F_{c,c'}$ denote the induced map of sets from $\Mor_\calc(c,c')$ to 
$\Mor_\cald(F(c),F(c'))$, and set $F_c=F_{c,c}$ for short.

\bigskip

\section{Higher limits over orbit categories}

Recall that when $\calf$ is a saturated fusion system over $S$, two 
subgroups of $S$ are said to be \emph{$\calf$-conjugate} if they are 
isomorphic in the category $\calf$. For example, for a finite group $G$ and 
$S\in\sylp{G}$, two subgroups are $\calf_S(G)$-conjugate if and only if 
they are $G$-conjugate in the usual sense. 

\begin{Defi} \label{d:F-centric} 
Let $\calf$ be a saturated fusion system over a finite $p$-group $S$. 
\begin{enuma} 
\item A subgroup $P\le S$ is \emph{$\calf$-centric} if for each $Q$ that 
is $\calf$-conjugate to $P$, $C_S(Q)\le Q$. 

\item Let $\calf^c$ denote the set of $\calf$-centric subgroups of $S$, and 
also (by abuse of notation) the full subcategory of $\calf$ with object set 
$\calf^c$. 

\item For each set $\X$ of subgroups of $S$, let $\calf^\X\subseteq\calf$ 
be the full subcategory with $\Ob(\calf^\X)=\X$. 

\end{enuma}
\end{Defi}

In general, we write $\homf(P,Q)$ for the set of $\calf$-morphisms from $P$ 
to $Q$. 

When $S$ is a $p$-group (in fact, any group), we let $\calt(S)$ denote the 
\emph{transporter category} of $S$: the category whose objects are the 
subgroup of $S$, and where 
	\[ \Mor_{\calt(S)}(P,Q) = T_S(P,Q) \defeq 
	\bigl\{g\in S\,\big|\, \9gP\le Q \bigr\}. \]
For any set $\X$ of subgroups of $S$, $\calt^\X(S)$ denotes the full 
subcategory of $\calt(S)$ with set of objects $\X$.

\begin{Defi}  \label{d:link}
Let $\calf$ be a fusion system over the $p$-group $S$, and let 
$\X\subseteq\calf^c$ be a nonempty family of subgroups closed under 
$\calf$-conjugacy and overgroups. An \emph{$\X$-linking system} associated 
to $\calf$ is a category $\call^\X$ with $\Ob(\call^\X)=\X$, together with 
functors $\pi\:\call^\X\too\calf^\X$ and $\delta\:\calt^\X(S)\too\call^\X$ 
that satisfy the following conditions.
\begin{enumerate}[\rm(A) ]
\renewcommand{\labelenumi}{\textup{(\Alph{enumi})}}%
\item Both $\delta$ and $\pi$ are the identity on objects, and $\pi$ is 
surjective on morphisms. For each $P,Q\in\X$, $Z(P)$ acts freely on 
$\Mor_{\call^\X}(P,Q)$ by composition (upon identifying $Z(P)$ with 
$\delta_P(Z(P))\le\Aut_{\call^\X}(P)$), and $\pi$ induces a bijection
	\[ \Mor_{\call^\X}(P,Q)/Z(P) \Right5{\cong} \homf(P,Q). \]

\item  For each $P,Q\in\X$ and each $g\in T_S(P,Q)$, $\pi$ sends 
$\delta_{P,Q}(g)\in\Mor_{\call^\X}(P,Q)$ to $c_g\in\Hom_{\calf}(P,Q)$.

\item  For each $P,Q\in\X$, $f\in\Mor_{\call^\X}(P,Q)$, and $g\in{}P$, 
	\[ f\circ \delta_P(g) = \delta_Q(\pi(f)(g))\circ f \in 
	\Mor_{\call^\X}(P,Q). \]

\end{enumerate}
Two $\X$-linking systems $\call_1^\X$ and $\call_2^\X$ associated to $\calf$ with 
structural functors $\pi_i\:\call_i^\X\too\calf^\X$ and 
$\delta_i\:\calt^\X(S)\too\call_i^\X$, are \emph{isomorphic} if there 
is an isomorphism $\Psi\:\call_1^\X\Right2{\cong}\call_2^\X$ of cate\-gories 
that commutes with the $\pi_i$ and the $\delta_i$.
\end{Defi}

Note that the set $\calf^c$ is closed under $\calf$-conjugacy and 
overgroups: the first holds by definition, and the second is easily 
checked. An $\calf^c$-linking system is exactly the same as a centric 
linking system as defined in \cite[\S\,III.4.1]{AKO}.

Aside from differences in requirements for the set of objects, this is the 
definition of a linking system given in \cite[Definition III.4.1]{AKO}), 
and is equivalent to the definition of a perfect locality in 
\cite[\S\S\,2.7--2.8]{Puig}. It is slightly different from the definition 
in \cite[Definition 1.7]{BLO2}, which for the purposes of comparison we 
call here a ``weak $\X$-linking system''.

\begin{Defi}  \label{d:weaklink}
Let $\calf$ be a fusion system over the $p$-group $S$, and let 
$\X\subseteq\calf^c$ be a nonempty family of subgroups closed under 
$\calf$-conjugacy and overgroups. A \emph{weak $\X$-linking system} 
associated to $\calf$ is a category $\call^\X$ with $\Ob(\call^\X)=\X$, 
together with a functor $\pi\:\call^\X\too\calf^\X$, and monomorphisms 
$\delta_P\:P\Right2{}\Aut_{\call^\X}(P)$ for each $P\in\X$, such that (A) 
and (C) in Definition \ref{d:link} both hold and (B) holds when $P=Q$ and 
$g\in P$. Two weak $\X$-linking systems $\call_1^\X$ and $\call_2^\X$ 
associated to $\calf$, with structural functors 
$\pi_i\:\call_i^\X\too\calf^\X$ and monomorphisms 
$(\delta_i)_P\:P\too\Aut_{\call_i^\X}(P)$ for $P\in\X$, are 
\emph{isomorphic} if there is an isomorphism of categories 
$\Psi\:\call_1^\X\Right2{\cong}\call_2^\X$ that commutes with the $\pi_i$ 
and with the $(\delta_i)_P$.
\end{Defi}

Most of the time, we just write ``$\call^\X$ is a (weak) $\X$-linking 
system'', and the functors $\pi$ and $\delta$ (or functions $\delta_P$) are 
understood. When we need to be more explicit, we write 
``$(\call^\X,\pi,\delta)$ is an $\X$-linking system'', or 
``$(\call^\X,\pi,\{\delta_P\})$ is a weak $\X$-linking system'' to include 
the structural functors (or functions) in the notation.

For $\calf$ and $\X$ as above, an $\X$-linking system 
$(\call^\X,\pi,\delta)$ restricts in an obvious way to a weak $\X$-linking 
system $(\call^\X,\pi,\{(\delta_0)_P\})$: just let 
$(\delta_0)_P\:P\too\Aut_{\call_0^\X}(P)$ be the restriction of 
$\delta_P\:\Aut_{\calt^\X(S)}(P)=N_S(P)\too \Aut_{\call^\X}(P)$ for each 
$P\in\X$. Note that for each $P\in\X$, 
$\Ker((\delta_0)_P)\le\Ker(\pi_P\circ(\delta_0)_P)=Z(P)$ by (B), so 
$(\delta_0)_P$ is a monomorphism by (A) ($Z(P)$ acts freely on 
$\Aut_{\call^X}(P)$). 
We also say that the $\X$-linking system $(\call^\X,\pi,\delta)$
extends $(\call^\X,\pi,\{(\delta_0)_P\})$ in this situation.

\begin{Prop} \label{uniq.ext.}
Let $\calf$ be a fusion system over the $p$-group $S$, and let 
$\X\subseteq\calf^c$ be a nonempty family of subgroups closed under 
$\calf$-conjugacy and overgroups. Then each weak $\X$-linking system 
$(\call^\X,\pi,\{(\delta_0)_P\})$ extends to an $\X$-linking system 
$(\call^\X,\pi,\delta)$, and any two such extensions are isomorphic as 
linking systems.
\end{Prop}

\begin{proof} The following property of (weak) linking systems is used 
repeatedly in the proof.
	\beqq \parbox{\short}{Let $P,Q,R\in\X$, 
	$\psi\in\Mor_{\call^\X}(P,R)$, $\psi_2\in\Mor_{\call^\X}(Q,R)$, and 
	$\varphi_1\in\homf(P,Q)$ be such that 
	$\pi_{Q,R}(\psi_2)\circ\varphi_1=\pi_{P,R}(\psi)$. Then there is a 
	unique morphism $\psi_1\in\Mor_{\call^\X}(P,Q)$ such that 
	$\pi_{P,Q}(\psi_1)=\varphi_1$ and $\psi_2\circ\psi_1=\psi$.} 
	\label{e:blo1.10} \eeqq
This is an easy consequence of condition (A), and is shown in \cite[Lemma 
1.10(a)]{BLO2}.

The existence of an $\X$-linking system $(\call^\X,\pi,\delta)$ that 
extends $(\call^\X,\pi,\{(\delta_0)_P\})$ is shown in \cite[Lemma 
1.11]{BLO2} (at least, when $\X=\calf^c$). We recall the construction here. 
For each $P\in\X$, choose an ``inclusion morphism'' 
$\iota_P\in\Mor_{\call^\X}(P,S)$ such that $\pi_{P,S}(\iota_P)=\incl_P^S$ (the 
inclusion of $P$ in $S$), and such that $\iota_S=\Id_S$. For each $P,Q\in\X$ 
and each $g\in T_S(P,Q)$, there is a unique element 
$\delta_{P,Q}(g)\in\Mor_{\call^\X}(P,Q)$ such that the following square 
commutes in $\call^\X$:
	\[ \xymatrix@C=50pt@R=40pt{
	P \ar[r]^{\iota_P} \ar[d]_{\delta_{P,Q}(g)} & S 
	\ar[d]_{\delta_S(g)} \\
	Q \ar[r]^{\iota_Q} & S \rlap{\,.}
	} \]
This is immediate by \eqref{e:blo1.10}, applied with $P,Q,S$ in the role 
of $P,Q,R$ and $c_g\in\homf(P,Q)$ in the role of $\varphi_1$. From the 
uniqueness in \eqref{e:blo1.10}, we also see that these morphisms combine 
to define a functor $\delta\:\calt^\X(S)\too\call^\X$. By condition (C) 
(and the uniqueness in \eqref{e:blo1.10}), $(\delta_0)_P(g)=\delta_P(g)$ 
for each $P\in\X$ and each $g\in P$. Thus $(\call^\X,\pi,\delta)$ is an 
$\X$-linking system that extends $(\call_0^\X,\pi_0,\{(\delta_0)_P\})$. 
Note also that $\iota_P=\delta_{P,S}(1)$ for each $P\in\X$.

Now let $\delta'$ be another functor such that $(\call^\X,\pi,\delta')$ is 
an $\X$-linking system that extends $(\call_0^\X,\pi_0,\{(\delta_0)_P\})$. For 
each $P\in\X$, set $\iota'_P=\delta'_{P,S}(1)$. Then 
$\pi_{P,S}(\iota'_P)=\incl_P^S=\pi_{P,S}(\iota_P)$ by condition (B), so 
by (A), there is $z_P\in Z(P)$ such that 
$\iota'_P=\iota_P\circ(\delta_0)_P(z_P)=\iota_P\circ\delta_P(z_P)$. For 
each $P,Q\in\X$ and $\psi\in\Mor_{\call^\X}(P,Q)$, consider the following 
diagram:
	\[ \xymatrix@C=50pt@R=15pt{
	P \ar[r]^{\delta_P(z_P)} \ar[drr]_(0.65){\iota'_P} 
	\ar[dd]_{\delta'_{P,Q}(g)} 
	& P \ar[dr]^{\iota_P} \ar[dd]_{\delta_{P,Q}(g)} \\
	&& S \ar[dd]^{\delta_S(g)=\delta'_S(g)} \\
	Q \ar[r]^{\delta_Q(z_Q)} \ar[drr]_{\iota'_Q} & Q \ar[dr]^{\iota_Q} \\
	&& S \rlap{\,.}
	} \]
Here, $\delta_S(g)=\delta'_S(g)$ since both are equal to 
$(\delta_0)_S(g)$ by assumption. The two parallelograms commute since 
$\delta$ and $\delta'$ are functors, and the two triangles commute by 
choice of $z_P$ and $z_Q$. Hence the square on the left 
commutes by the uniqueness in \eqref{e:blo1.10}. So if we define a functor 
$\Theta\:\call^\X\too\call^\X$ by setting $\Theta(P)=P$ for $P\in\X$ and 
$\Theta(\psi)=\delta_Q(z_Q)\circ\psi\circ\delta_P(z_P)^{-1}$ for 
$\psi\in\Mor_{\call^\X}(P,Q)$, then $\Theta\circ\delta'=\delta$ and 
$\pi\circ\Theta=\pi$. Thus $\Theta$ is an isomorphism from 
$(\call^\X,\pi,\delta')$ to $(\call^\X,\pi,\delta)$.
\end{proof}

Since isomorphic linking systems clearly restrict to isomorphic weak 
linking systems, Proposition \ref{uniq.ext.} shows that for $\calf$ and 
$\X$ as above, there is a natural bijection between the set of isomorphism 
classes of $\X$-linking systems associated to $\calf$ and the set of 
isomorphism classes of weak $\X$-linking systems associated to $\calf$. In 
particular, the obstruction theory set up in \cite[\S\,3]{BLO2} for the 
existence and uniqueness of weak linking systems also applies to that for 
linking systems in the sense of Definition \ref{d:link} (see Proposition 
\ref{obstr}).

We next define orbit categories, since they play an important role here. In 
fact, we need to consider two different types of orbit categories: those 
for fusion systems and those for groups. 

\begin{Defi} \label{d:orbitcat}
\begin{enuma} 

\item Let $\calf$ be a saturated fusion system over a finite $p$-group $S$. 
The \emph{orbit category} $\calo(\calf)$ of $\calf$ is the category with 
the same objects (the subgroups of $S$), and such that for each $P,Q\le S$,
	\[ \Mor_{\calo(\calf)}(P,Q) = Q{\backslash}\homf(P,Q). \]
Here, $g\in Q$ acts on $\homf(P,Q)$ by post-composition with 
$c_g\in\Inn(Q)$. Thus a morphism in $\calo(\calf)$ is a conjugacy class of 
morphisms in $\calf$. Also,
let $\calo(\calf^c)\subseteq\calo(\calf)$ be the full subcategory with 
object set $\calf^c$.

\item Let $G$ be a finite group, and fix $S\in\sylp{G}$. Let $\calo_S(G)$ 
be the category where $\Ob(\calo_S(G))$ is the set of subgroups of $S$, and 
where
	\[ \Mor_{\calo_S(G)}(P,Q) = \textup{map}_G(G/P,G/Q)\,: \]
the set of $G$-equivariant maps from the transitive $G$-set $G/P$ to the 
$G$-set $G/Q$. Note that each $\varphi\:G/P\too G/Q$ has the form 
$\varphi(gP)=gaQ$ (for all $g\in G$) for some fixed $a\in G$ such that 
$P\le\9aQ$.
\end{enuma}
\end{Defi}

Note that for a finite group $G$ and $S\in\sylp{G}$, there is a natural 
surjective functor
	\[ \calo_S(G) \Right4{} \calo(\calf_S(G)) \,: \]
this is the identity on objects, and sends a morphism $(gP\mapsto gaQ)$ 
(from $G/P$ to $G/Q$) to the class of $c_a^{-1}\in\Hom_{\calf_S(G)}(P,Q)$.

We refer to \cite[\S\,III.5.1]{AKO} for a very brief summary of some basic 
properties of ``higher limits'': higher derived functors of inverse limits. 
We also refer to \cite[\S\S\,5--6]{JMO} for more details about higher 
limits over orbit categories of groups, to \cite[\S\,3]{BLO2} for those 
over orbit categories of fusion systems, and to \cite[\S\,III.5.4]{AKO} for 
both.

When $\calf$ is a saturated fusion system over a finite $p$-group $S$, and 
$\X\subseteq\calf^c$ is closed under $\calf$-conjugacy and overgroups, 
define 
	\[ \calz_\calf^\X \: \calo(\calf^c)\op \Right4{} \Ab 
	\qquad\textup{by setting}\qquad
	\calz_\calf^\X(P) = \begin{cases} 
	Z(P)=C_S(P) & \textup{if $P\in\X$} \\
	0 & \textup{if $P\notin\X$.}
	\end{cases} \]
When $P,Q\in\X$, $\calz_\calf^\X$ sends a morphism $(P\xto{[\varphi]}Q)$ to 
$\bigl(Z(P)\Right2{\varphi^{-1}}Z(Q)\bigr)$ (where 
$[\varphi]\in\Mor(\calo(\calf^c))$ is the class of 
$\varphi\in\Mor(\calf^c)$). If 
$\Y\subseteq\X\subseteq\calf^c$ are both closed under $\calf$-conjugacy and 
overgroups, it is not hard to see that $\calz_\calf^\Y$ is a quotient 
functor of $\calz_\calf^\X$.

\begin{Prop} \label{obstr}
Let $\calf$ be a saturated fusion system over a finite $p$-group $S$. Let 
$\X\subseteq\calf^c$ be a nonempty family of subgroups that is closed under 
$\calf$-conjugacy and overgroups, and let $\linkset_\calf^\X$ be the set of 
all isomorphism classes of $\X$-linking systems associated to $\calf$.
\begin{enuma} 

\item The set $\linkset_\calf^\X$ is nonempty if and only if a certain 
obstruction in $\higherlim3{\calo(\calf^c)}{\calz_\calf^\X}$ is zero. In 
particular, $\linkset_\calf^\X\ne\emptyset$ whenever 
$\higherlim3{\calo(\calf^c)}{\calz_\calf^\X}=0$.

\item If $\linkset_\calf^\X\ne\emptyset$, then the group 
$\higherlim2{\calo(\calf^c)}{\calz_\calf^\X}$ acts freely and transitively 
on $\linkset_\calf^\X$, and hence has the same cardinality as 
$\linkset_\calf^\X$.

\end{enuma}
\end{Prop}

\begin{proof} This follows with exactly the same proof as that of 
\cite[Proposition 3.1]{BLO2} (the case where $\X=\calf^c$).
\end{proof}

Now fix a finite group $\Gamma$ and a $\Z_{(p)}\Gamma$-module $M$. Choose 
$T\in\sylp\Gamma$, and let 
	\beqq F_M\:\calo_T(\Gamma)\op\Right4{}\Ab 
	\qquad\textup{be defined by}\qquad
	F_M(P) = \begin{cases} 
	M & \textup{if $P=1$} \\
	0 & \textup{if $P\ne1$.} 
	\end{cases} \label{e:F_M} \eeqq
Here, $\Aut_{\calo_T(\Gamma)}(1)\cong\Gamma$ has the given action on 
$F_M(1)=M$. For each $i\ge0$, set
	\[ \Lambda^i(\Gamma;M) = \higherlim{i}{\calo_T(\Gamma)}{F_M}. \]


\begin{Thm} \label{Lambda3}
Fix a finite group $\Gamma$, and an $\F_p\Gamma$-module $M$ on which 
$\Gamma$ acts faithfully. Set $G=M\rtimes\Gamma$, choose $T\in\sylp\Gamma$, 
and set $S=M\rtimes T\in\sylp{G}$. Set 
	\[ \X = \{P\le S\,|\,P\ge M\} \qquad\textup{and}\qquad
	\Y = \{P\le S\,|\,P>M\}=\X\sminus\{M\}\,. \]
Then 
\begin{enuma} 

\item there is a unique isomorphism class of $\X$-linking systems 
associated to $\calf_S(G)$; and 

\item the set of isomorphism classes of $\Y$-linking systems associated to 
$\calf_S(G)$ is in bijective correspondence with $\Lambda^3(\Gamma;M)$.

\end{enuma}
Thus if $\Lambda^3(\Gamma;M)\ne0$, then there is (up to isomorphism) more 
than one $\Y$-linking system associated to $\calf$, only one of which can 
be extended to an $\X$-linking system.
\end{Thm}

\begin{proof} By \cite[Lemma 1.6(b)]{O-Ch}, 
	\beqq \higherlim{i}{\calo(\calf^c)}{\calz_\calf^\X}=0 \qquad 
	\textup{for all $i>0$.} \label{e:Z_X} \eeqq
In particular, by Proposition \ref{obstr}, there is up to isomorphism a 
unique $\X$-linking system $\call^\X$ associated to $\calf$. This also 
shows that there is at least one $\Y$-linking system: the full subcategory 
of $\call^\X$ with object set $\Y$.

Let $\calz_0\subseteq\calz_\calf^\X$ be the subfunctor 
	\[ \calz_0(P) = \begin{cases} 
	0 & \textup{if $P\in\Y$} \\
	\calz_\calf^\X(P)=M & \textup{if $P=M$.}
	\end{cases} \]
Thus $\calz_\calf^\X/\calz_0\cong\calz_\calf^\Y$. By \cite[Proposition 
3.2]{BLO2}, for each $i\ge0$, 
	\[ \higherlim{i}{\calo(\calf^c)}{\calz_0} \cong 
	\Lambda^i(\outf(M);\calz_0(M)) \cong \Lambda^i(\Gamma;M). \]
So from \eqref{e:Z_X} and the long exact sequence of higher limits for the 
extension 
	\[ 0 \Right2{} \calz_0 \Right4{} \calz_\calf^\X \Right4{} 
	\calz_\calf^\Y \Right2{} 0 \]
(see, e.g., \cite[Proposition 5.1(i)]{JMO} or \cite[Lemma 1.7]{O-Ch}), we get 
that 
	\beqq \higherlim{i}{\calo(\calf^c)}{\calz_\calf^\Y}\cong 
	\Lambda^{i+1}(\Gamma;M) \qquad \textup{for all 
	$i>0$.} \label{e:Z_Y} \eeqq
By Proposition \ref{obstr} again, the $\Y$-linking 
systems associated to $\calf$ are in bijective correspondence with 
$\Lambda^3(\Gamma;M)$. 
\end{proof}

\bigskip

\section{Some explicit examples}

We now give some concrete examples of pairs $(\Gamma,M)$ such that 
$\Lambda^3(\Gamma;M)\ne0$, using three independent methods. 

In general, if $M$ is an $\F_p\Gamma$-module such that 
$\Lambda^k(\Gamma;M)\ne0$ for some $k\ge1$, then $\dim_{\F_p}(M)\ge p^k$ 
(see \cite[Proposition 6.3]{BLO1} or \cite[Lemma III.5.27]{AKO}). This 
helps to explain why the examples given below (for $k=3$) are fairly large: 
there are no examples when $\dim(M)<p^3$. In fact, in the examples of 
\ref{ex2} and \ref{ex3}, $M$ has dimension exactly $p^3$.

\begin{Ex} \label{ex1}
Let $p$ be any prime, and let $\Gamma$ be a finite group of Lie type of Lie 
rank $3$ in defining characteristic $p$. Let $\St(\Gamma)$ be the Steinberg 
module (over $\F_p$) for $\Gamma$. Then 
$\Lambda^3(\Gamma;\St(\Gamma))\cong\F_p$. 
\end{Ex}

\begin{proof} Fix $U\in\sylp\Gamma$. By a theorem of Grodal \cite[Theorem 
4.1]{Grodal}, and since $\Gamma$ has Lie rank $3$,
	\[ \Lambda^3(\Gamma;\St(\Gamma)) = 
	\higherlim3{\calo_U(\Gamma)}{F_{\St(\Gamma)}} 
	\cong \Hom_\Gamma(\St(\Gamma),\St(\Gamma)) \cong \F_p, \]
where the last isomorphism holds since $\St(\Gamma)$ is absolutely 
irreducible (see \cite[Proposition 6.2.2]{Carter}). 
\end{proof}

For example, in Example \ref{ex1}, we can take $\Gamma=\SL_4(p)$, and 
let $M$ be its $p^6$-dimensional Steinberg module \cite[Corollary 
6.4.3]{Carter}.

We next list some of the elementary properties of the $\Lambda^*(\Gamma;M)$ 
that will be used in the other two examples.

\begin{Prop} \label{Lambda-props}
Fix a finite group $\Gamma$ and a $\Z_{(p)}\Gamma$-module $M$.
\begin{enuma} 

\item $\Lambda^0(\Gamma;M)=0$ if $p\bmid|\Gamma|$, and 
$\Lambda^0(\Gamma;M)\cong C_M(\Gamma)$ otherwise.

\item If $p\bmid|C_\Gamma(M)|$, or if $O_p(\Gamma)\ne1$, then 
$\Lambda^i(\Gamma;M)=0$ for all $i\ge0$.

\item \textup{(Künneth formula)} If $\Gamma_1$ and $\Gamma_2$ are two 
finite groups, and $M_i$ is a finitely generated $\F_p\Gamma_i$-module for 
$i=1,2$, then for each $k\ge0$,
	\[ \Lambda^k(\Gamma_1\times\Gamma_2;M_1\otimes_{\F_p} M_2) \cong 
	\bigoplus_{j=0}^k \Lambda^j(\Gamma_1;M_1) \otimes_{\F_p} 
	\Lambda^{k-j}(\Gamma_2;M_2). \]

\item If $T\in\sylp\Gamma$ has order $p$, then $\Lambda^1(\Gamma;M)\cong 
C_M({N_\Gamma(T)})\big/C_M(\Gamma)$, and $\Lambda^i(\Gamma;M)=0$ for $i\ne1$.

\end{enuma}
\end{Prop}

\begin{proof} See \cite[Propositions 6.1(i,ii,v) \& 6.2(i)]{JMO}, 
respectively. 
\end{proof}

As one easy application of Proposition \ref{Lambda-props}(d), if 
$V\cong(\F_p)^p$ is the natural module for $\Sigma_{p+1}$ over $\F_p$ 
(i.e., the $(p+1)$-dimensional permutation module modulo the diagonal), then 
	\beqq \Lambda^i(\Sigma_{p+1};V)\cong \begin{cases} 
	\F_p & \textup{if $i=1$} \\
	0 & \textup{if $i\ne1$.}
	\end{cases} \label{e:Lambda1-ex} \eeqq
This will be used in each of the next two examples below.


\begin{Ex} \label{ex2}
For each prime $p$, 
	\[ \Lambda^3(\Sigma_{p+1}\times\Sigma_{p+1}\times\Sigma_{p+1}
	\,;\, V\otimes V\otimes V) \cong\F_p, \]
where $V\cong(\F_p)^p$ is the natural module for $\Sigma_{p+1}$.
\end{Ex}

\begin{proof} This follows from \eqref{e:Lambda1-ex} and the Künneth 
formula (Proposition \ref{Lambda-props}(c)). 
\end{proof}

The last example is based on taking wreath products with $C_p$, using the 
following formula.

\begin{Lem} \label{wreath}
Let $\Gamma$ be a finite group such that $p\bmid|\Gamma|$. Then for each 
$\F_p\Gamma$-module $M$ and each $i\ge1$,
	\[ \Lambda^i(\Gamma\wr C_p;M^p) \cong \Lambda^{i-1}(\Gamma;M)\,. 
	\]
\end{Lem}

\begin{proof} Set $G=\Gamma\wr C_p$ for short. Let $G_0\nsg\Gamma$ and 
$x\in G\sminus G_0$ be such that $G_0=\Gamma^p$ (a fixed identification), 
$x^p=1$, and $\9x(g_1,\dots,g_p)=(g_2,\dots,g_p,g_1)$. For 
$g=(g_1,\dots,g_p)\in G_0$, $(gx)^p=1$ if and only if $g_1g_2\cdots g_p=1$, 
in which case $gx$ is $G$-conjugate to $x$. 

Fix $T\in\sylp\Gamma$, and set $S=T^p\gen{x}\in\sylp{G}$. Define 
$\N\:\calo_S(G)\op\Right2{}\F_p\mod$ by setting 
	\[ \N(P) = \bigl\{ \textstyle\sum_{g\in P}g\xi \,\big|\, \xi\in M^p \bigr\}. \] 
If $\varphi\in\Mor_{\calo_S(G)}(P,Q)=\map_G(G/P,G/Q)$ has the form 
$\varphi(gP)=gaQ$ for some $a\in G$ such that $P\le \9aQ$, then 
$\N(\varphi)\:\N(Q)\too\N(P)$ is defined by setting $\N(\varphi)(\xi)=a\xi$. 
By \cite[Proposition 5.2]{JMO} (recall that $\5H^0(P;M)\cong C_M(P)/\N(P)$), or 
(more explicitly) by \cite[Proposition 1.7]{O-limz},
	\beqq \higherlim{i}{\calo_S(G)}{\N} \cong \begin{cases} 
	\N(G) \defeq
	\bigl\{ \sum_{g\in G}g\xi \,\big|\, \xi\in M^p \bigr\} = 0 & 
	\textup{if $i=0$} \\
	0 & \textup{if $i\ge1$.}
	\end{cases} \label{e:1a} \eeqq
(Recall that $p\bmid|\Gamma|$ and $pM=0$ when checking that $\N(G)=0$.)

Let $F_{M^p}$ be as in \eqref{e:F_M}, regarded as a subfunctor of $\N$, and 
set $\N_0=\N/F_{M^p}$. Thus $\N_0(P)=\N(P)$ for $1\ne P\le S$ and 
$\N_0(1)=0$. By \eqref{e:1a} and the long exact sequence for the extension 
$0\to F_{M^p}\too\N\too\N_0\to0$ of functors, for each $i>0$,
	\beqq \Lambda^i(G;M^p) = \higherlim{i}{\calo_S(G)}{F_{M^p}} 
	\cong \higherlim{i-1}{\calo_S(G)}{\N_0}. \label{e:2} \eeqq

Now fix $1\ne P\le S$, and set $P_0=P\cap G_0$. Assume that 
$\Lambda^*(N_G(P)/P;\N_0(P))\ne0$. For $1\le i\le p$, let $P_i$ be 
the image of $P_0$ under projection to the $i$-th factor of 
$G_0=\Gamma^p$, and set $\5P=P_1\times\cdots\times P_p\in G_0$. Each element in 
$N_G(P)$ normalizes $P_0$ and hence normalizes $\5P$, so $P\5P\le G$, and 
$N_{P\5P}(P)\nsg N_G(P)$. If $\5P>P_0$, then $P\5P>P$, and $1\ne 
N_{P\5P}(P)/P\le O_p(N_G(P)/P)$. This contradicts Proposition 
\ref{Lambda-props}(b), and thus $P_0=\5P=P_1\times\cdots\times P_p$. 

If two or more of the $P_i$ are nontrivial, then $\N(P)=0$. Otherwise, we 
can assume (up to conjugacy in $G$) that $P_i=1$ for $2\le i\le p$. If 
$P_1\ne1$, then $P=P_0$, $\N(P)\le M\times0\times\cdots0$, and hence 
$1\times\Gamma^{p-1}\le N_G(P)$ acts trivially on $\N(P)$. Since 
$p\bmid|\Gamma|$, this again contradicts Proposition \ref{Lambda-props}(b). 
Hence $P_0=1$, $P\ne1$, and $P$ is $G$-conjugate to $\gen{x}$ by the earlier 
remarks. 

Thus for $P\le S$, $\Lambda^*(N_G(P)/P;\N_0(P))=0$ except when $P$ is 
$G$-conjugate to $\gen{x}$. So by \eqref{e:2} and \cite[Corollary 
III.5.21(b)]{AKO}, 
	\beq \Lambda^i(G;M^p) \cong \higherlim{i-1}{\calo_S(G)}{\N_0}
	\cong \Lambda^{i-1}(N_G(\gen{x})/\gen{x};\N_0(\gen{x})) \cong
	\Lambda^{i-1}(\Gamma;M). \qedhere \eeq
\end{proof}

The third example now follows immediately from \eqref{e:Lambda1-ex} and 
Lemma \ref{wreath}.

\begin{Ex} \label{ex3}
For each prime $p$,
	\[ \Lambda^3(\Sigma_{p+1}\wr C_p\wr C_p \,\,;\, V^{p^2})\cong \F_p, \]
where $V\cong(\F_p)^p$ is the natural module for $\Sigma_{p+1}$. \qed
\end{Ex}

\bigskip\bigskip

\begin{thebibliography}{BLO2}

\bibitem[AKO]{AKO} M. Aschbacher, R. Kessar, \& B. Oliver, Fusion systems 
in algebra and topology, Cambridge Univ. Press (2011)

\bibitem[BLO1]{BLO1} C. Broto, R. Levi, \& B. Oliver, Homotopy 
equivalences of $p$-completed classifying spaces of finite groups, 
\emph{Invent. Math.}, {\bf151} (2003), 611--664.

\bibitem[BLO2]{BLO2} C. Broto, R. Levi, \& B. Oliver, The homotopy theory 
of fusion systems, Journal Amer. Math. Soc. 16 (2003), 779--856

\bibitem[Ca]{Carter} R. Carter, Finite groups of Lie type: conjugacy 
classes and complex characters, Wiley (1985)

\bibitem[Ch]{Chermak} A. Chermak, Fusion systems and localities, Acta Math. 
211 (2013), 47--139.

\bibitem[GbL]{GLynd} G. Glauberman \& J. Lynd, Control of weak closure and 
existence and uniqueness of centric linking systems, Invent. Math. 206 
(2016), 441--484

\bibitem[Gr]{Grodal} J. Grodal, Higher limits via subgroup complexes, 
\emph{Annals of Math.}, {\bf155} (2002), 405--457.

\bibitem[JMO]{JMO} S. Jackowski, J. McClure, \& B. Oliver, Homotopy 
classification of self-maps of $BG$ via $G$-actions, \emph{Annals of 
Math.}, {\bf135} (1992), 183--270

\bibitem[O1]{O-limz} B. Oliver, Equivalences of classifying spaces 
completed at the prime two, Memoirs Amer. Math. Soc. 848 (2006)

\bibitem[O2]{O-Ch} B. Oliver, Existence and uniqueness of linking systems: 
Chermak's proof via obstruction theory, Acta Math. 211 (2013), 141--175

\bibitem[P1]{Puig1} L. Puig, Frobenius categories, J. Algebra 303 (2006),
309--357

\bibitem[P2]{Puig} L. Puig, Existence, uniqueness, and functoriality of the 
perfect locality over a Frobenius $P$-category, Algebra Coll. 23 (2016), 
541--622


\end{thebibliography}
\end{document}